\documentclass[12pt]{article}
\usepackage{amssymb,latexsym,amsmath}
\usepackage{amsthm}
\usepackage{geometry}
\newtheorem{thm}{Theorem}

\newtheorem{lem}[thm]{Lemma}

\theoremstyle{plain}

\theoremstyle{definition}

\numberwithin {equation}{section}
\usepackage{longtable}
\usepackage{rotating}

\begin{document}
\title
{Solving nonlinear equations by a\\ derivative-free form of the
King's family with memory }
\author{Somayeh Sharifi$^{a}$\thanks{s.sharifi@iauh.ac.ir} \and  Stefan Siegmund$^{b}$\thanks{stefan.siegmund@tu-dresden.de} \and Mehdi Salimi
$^{c}$\thanks{Corresponding author: mehdi.salimi@tu-dresden.de}}
\date{}
\maketitle
\begin{center}
$^{a}$Young Researchers and Elite Club, Hamedan Branch, Islamic
Azad University, Hamedan, Iran\\
$^{b,c}$Center for Dynamics, Department of Mathematics, Technische
Universit{\"a}t Dresden, 01062 Dresden, Germany\\
\end{center}
\maketitle

\begin{abstract}
In this paper, we present an iterative three-point method with
memory based on the family of King's methods to solve nonlinear
equations. This proposed method has eighth order convergence and
costs only four function evaluations per iteration which supports
the Kung-Traub conjecture on the optimal order of convergence. An
acceleration of the convergence speed is achieved by an
appropriate variation of a free parameter in each step. This self
accelerator parameter is estimated using Newton's interpolation
polynomial of fourth degree. The order of convergence is increased
from 8 to 12 without any extra function evaluation. Consequently,
this method, possesses a high computational efficiency. Finally, a
numerical comparison of the
proposed method with related methods shows its effectiveness and performance in high precision computations.
2010 Mathematics Subject Classification, 65H05.
 \end{abstract}
\section{Introduction}
Solving nonlinear equations is one of the most important problems
that has interesting applications in all fields in science and
engineering. We consider iterative methods to find a simple root
$\alpha$ of a nonlinear equation $f(x)=0$, where $f : D \subset
\mathbb{R} \to \mathbb{R}$ for an open interval $D$ is a scalar
function. Newton-Raphson iteration $x_{n+1} = x_n -
\frac{f(x_n)}{f'(x_n)}$ is probably the most widely used algorithm
for finding roots. It is of second order and requires two
evaluations for each iteration step, one evaluation of $f$ and one
of $f'$ \cite{Ostrowski,Traub}.

Kung and Traub \cite{Kung} conjectured that no multi-point method
without memory with $n$ evaluations could have a convergence order
larger than $2^{n-1}$. A multi-point method with convergence order
$2^{n-1}$ is called optimal. The efficiency index, gives a measure
of the balance between those quantities, according to the formula
$p^{1/n}$, where $p$ is the convergence order of the method and
$n$ is the number of function evaluations per iteration.

Some well-known optimal two-point methods have been introduced by
Jarratt \cite{Jarrat}, King \cite{King} and Ostrowski
\cite{Ostrowski}. Some optimal three-point methods have been
proposed by Chun and Lee \cite{Chun1}, Cordero et al.\
\cite{Cordero22,Cordero5}, Lotfi et al. \cite{Lotfi}, Neta
\cite{Neta0}, Salimi et al. \cite{Salimi} and Sharma et al.\
\cite{Sharma1}. Zheng et al.\ \cite{Zheng} have presented an
optimal Steffensen-type family and Dzunic et al.\ \cite{Dzunic},
Petkovic \cite{Pet} and Sharma \cite{Sharma2} have constructed
methods with memory to solve nonlinear equations.

In this paper, we present a modification of the family of King's
methods which is derivative-free. The result of this paper is organized as
follows: Section 2 and 3 are devoted to the construction and
convergence analysis of two-point and three-point optimal derivative-free
methods. We introduce the methods with memory and prove their
R-order in Section 4. In Section 5, the new methods are
compared with a closest competitor in a series of numerical
examples. Section 6 contains a short conclusion.

\section{Derivative-free two-point method of fourth order}
\subsection{Description of derivative-free two-point method}
We start with King's family of methods which is one of the most
important two-point families for solving nonlinear equations
\cite{King}.
\begin{equation}\label{a1}
\begin{cases}
y_n=x_n-\dfrac{f(x_n)}{f'(x_n)},\\
x_{n+1}=y_n-\dfrac{f(y_n)}{f'(x_n)}\cdot\dfrac{f(x_n)+\gamma
f(y_n)}{f(x_n)+(\gamma-2)f(y_n)},\quad (n=0, 1, \ldots),\quad
\gamma\in \mathbb{R},
\end{cases}
\end{equation}
where $x_0$ is an initial approximation of a simple zero $\alpha$ of
$f$. The main idea is to construct a derivative-free class of
two-point methods with optimal order of convergence four. We
consider Steffensen-Like's method for the first step and for the
second step approximate $f'(x_n)$ by
\begin{equation*}
f'(x_n)\approx\dfrac{f[y_n,w_n]}{G(t_n)},
\end{equation*}
where $ w_n=x_n-\beta f(x_n)$, $\beta \neq 0$,
$f[y_n,w_n]=\dfrac{f(y_n)-f(w_n)}{y_n-w_n}$,
$t_n=\dfrac{f(y_n)}{f(x_n)}$ and $G$ is a real function.\\
Hence, we obtain
\begin{equation}\label{a2}
\begin{cases}
y_n=x_n-\dfrac{\beta
f(x_n)^2}{f(x_n)-f(w_n)}, \\
x_{n+1}=y_n-\dfrac{f(x_n)+\gamma
f(y_n)}{f(x_n)+(\gamma-2)f(y_n)}\cdot\dfrac{f(y_n)}{f[y_n,w_n]}G(t_n).
\end{cases}
\end{equation}
\subsection{Convergence analysis}
We shall state the convergence theorem for the family of methods
(\ref{a2}).\\
\begin{thm}\label{t1}
Let $D \subseteq \mathbb{R}$ be an open interval, $f : D
\rightarrow \mathbb{R}$ four times continuously differentiable and
let $\alpha\in D$ be a simple zero of $f$. If the initial point
$x_{0}$ is sufficiently close to $\alpha$, then the method defined
by (\ref{a2}) converges to $\alpha$ with order four if the weight
function $G : \mathbb{R} \rightarrow \mathbb{R}$ is
continuously differentiable and satisfies the conditions
\begin{equation*}
G(0)=1, \quad \quad G'(0)=2 \gamma-1.
\end{equation*}
\end{thm}
\begin{proof}
Let $e_{n}:=x_{n}-\alpha$, $e_{n,w}:=w_n-\alpha$,
$e_{n,y}:=y_{n}-\alpha$ and
$c_{n}:=\dfrac{f^{(n)}(\alpha)}{n!f^{'}(\alpha)}, \quad n=2,3,\ldots$.\\
Using Taylor expansion of $f$ at $\alpha$ and taking into
account $f(\alpha)=0$, we have
\begin{equation}\label{a3}
f(x_{n}) = f^{'}(\alpha)(e_{n} +
c_{2}e_{n}^{2}+c_{3}e_{n}^{3}+c_{4}e_{n}^{4})+O(e_n^5),
\end{equation}
then
\begin{equation*}\
e_{n,w}=w_n-\alpha=(1-\beta f'(\alpha))e_n-\beta
f'(\alpha)c_2e_n^2+O(e_n^3),
\end{equation*}
and
\begin{equation}\label{a4}
f(w_{n})=f^{'}(\alpha)(e_{n,w} +
c_{2}e_{n,w}^{2}+c_{3}e_{n,w}^{3}+c_{4}e_{n,w}^{4})+O(e_{n,w}^5).
\end{equation}
From (\ref{a3}) and (\ref{a4}), we have
\begin{equation}\label{a5}
\dfrac{\beta f(x_n)^2}{f(x_n)-f(w_n)}=e_n-(1-\beta
f'(\alpha))c_2e_n^2+O(e_n^3),
\end{equation}
by substituting (\ref{a5}) in (\ref{a2}), we get
\begin{equation*}
e_{n,y}=y_n-\alpha=(1-\beta f'(\alpha))c_2e_n^2+O(e_n^3).
\end{equation*}
As well as
\begin{equation}\label{a6}
f(y_{n})=f^{'}(\alpha)(e_{n,y} +
c_{2}e_{n,y}^{2}+c_{3}e_{n,y}^{3}+c_{4}e_{n,y}^{4})+O(e_{n,y}^5).
\end{equation}
From (\ref{a3}) and (\ref{a6}), we obtain
\begin{equation}\label{a7}
\begin{split}
t_{n}&=\dfrac{f(y_n)}{f(x_n)}=(1-\beta f'(\alpha))c_2e_n\\
&+\left(-(3+\beta f'(\alpha)(-3+\beta f'(\alpha)))c_2^2+(-2+\beta
f'(\alpha))(-1+\beta f'(\alpha))c_3\right)e_n^2+O(e_n^3),
\end{split}
\end{equation}
by expanding $G(t_n)$ around $0$, we have
\begin{equation}\label{a8}
G(t_n)=G(0)+G'(0)t_n+O(t_n^2),
\end{equation}
moreover, from (\ref{a4}) and (\ref{a6}), we obtain
\begin{equation}\label{a9}
\begin{split}
f[y_n,w_n]&=\dfrac{f(y_n)-f(w_n)}{y_n-w_n}
=f'(\alpha)+f'(\alpha)(1-\beta f'(\alpha))c_2e_n\\
&+f'(\alpha)((1-2\beta f'(\alpha))c_2^2+(-1+\beta
f'(\alpha))^2c_3)e_n^2+O(e_n^3),
\end{split}
\end{equation}
then, by substituting (\ref{a3})-(\ref{a9}) in (\ref{a2}) we get
\begin{equation*}
e_{n+1}=z_n-\alpha=R_2e_n^2+R_3e_n^3+R_4e_n^4+O(e_n^5),
\end{equation*}
with
\begin{align*}
R_2&=-c_2(1-\beta f'(\alpha))(-1+G(0)), \\
R_3&=c_2^2(1-\beta f(\alpha))^2(-1+2\gamma-G'(0)),\\
R_4&=-c_2(1-\beta f'(\alpha))^2(c_2^2(-1-\beta
f'(\alpha)-2\gamma^2(1-\beta f'(\alpha)))+c_3).
\end{align*}
Therefore, to provide the fourth order convergence of the
two-point method (\ref{a2}), it is necessary to choose $R_i=0$ for $i=2,3$, and to achieve this we use the fact that
\begin{align*}
R_2&=0 \quad \text{if} \quad G(0)=1,\\
R_3&=0 \quad \text{if} \quad G'(0)=2\gamma-1.
\end{align*}
It is clear that $R_{4}\neq 0$ in general. Thus, method (\ref{a2}) converges
to $\alpha$ with order four and the error equation becomes
\begin{equation}
e_{n+1}=-c_2\left(1-\beta f'(\alpha)\right)^2\left(c_2^2(-1-\beta
f'(\alpha)-2\gamma^2(1-\beta
f'(\alpha)))+c_3\right)e_n^4+O(e_n^5).
\end{equation}
This finishes the proof of the theorem.
\end{proof}

\section{Derivative-free three-point method with order eight}

\subsection{Description of derivative-free three-point method}

Again, we will add one Newton step to the method (\ref{a2})\\
\begin{equation}\label{b1}
\begin{cases}
y_n=x_n-\dfrac{\beta f(x_n)^2}{f(x_n)-f(w_n)},\\
z_n=y_n-\dfrac{f(x_n)+\gamma
f(y_n)}{f(x_n)+(\gamma-2)f(y_n)}\dfrac{f(y_n)}{f[y_n,w_n]}G(t_n),\\
x_{n+1}=z_n-\dfrac{f(z_n)}{f'(z_n)}.
\end{cases}
\end{equation}
As it is seen the function is evaluated for five times, so this
method is not optimal and it is also not free of derivatives. To make it an optimal and
derivative-free method, we approximate $ f'(z_n) $ with Newton's
interpolation polynomial of degree three at the point $x_n$,
$w_n$, $y_n$ and $z_n$.
\begin{align*}
N_3(t;z_n,y_n,x_n,w_n)&:=f(z_n)+f[z_n,y_n](t-z_n)\\
&+f[z_n,y_n,x_n](t-z_n)(t-y_n)+f[z_n,y_n,x_n,w_n](t-z_n)(t-y_n)(t-x_n).
\end{align*}
It is clear that
\begin{equation*}
N_3(z_n)=f(z_n), \quad \text{and} \quad
N_3'(t)\mid_{t=z_n}=f'(z_n).
\end{equation*}
Then
\begin{equation*}
N_3'(z_n)=\left[\frac{d}{dt}N_3(t)\right]_{t=z_n}=f[z_n,y_n]+f[z_n,y_n,x_n](z_n-y_n)+f[z_n,y_n,x_n,w_n](z_n-y_n)(z_n-x_n),
\end{equation*}
and hence we get
\begin{equation}\label{b2}
\begin{cases}
y_n=x_n-\dfrac{\beta f(x_n)^2}{f(x_n)-f(w_n)},\\
z_n=y_n-\dfrac{f(x_n)+\gamma
f(y_n)}{f(x_n)+(\gamma-2)f(y_n)}\dfrac{f(y_n)}{f[y_n,w_n]}G(t_n),\\
x_{n+1}=z_n-\dfrac{f(z_n)}{f[z_n,y_n]+f[z_n,y_n,x_n](z_n-y_n)+f[z_n,y_n,x_n,w_n](z_n-y_n)(z_n-x_n)}.\\
\end{cases}
\end{equation}

\subsection{Convergence analysis}
The following theorem shows that the method (\ref{b2}) has
convergence order
eight.
\begin{thm}\label{t2}
Let $D \subseteq \mathbb{R}$ be an open interval, $f : D
\rightarrow \mathbb{R}$ eight times continuously differentiable
and let $ \alpha\in D$ be a simple zero of $f$. If the initial
point $x_{0}$ is sufficiently close to $\alpha$ and conditions of
Theoerem \ref{t1} are established, then the method defined by
(\ref{b2}) converges to $\alpha$ with order eight.
\end{thm}
\begin{proof}
Define $e_{n}:=x_{n}-\alpha$, $e_{n,w}:=w_n-\alpha$,
$e_{n,y}:=y_{n}-\alpha$, $e_{n,z}:=z_{n}-\alpha$ and
$c_{n}:=\frac{f^{(n)}(\alpha)}{n!f^{'}(\alpha)}$ for $n=2,3,\ldots$.

Using Taylor expansion of $f$ at $\alpha$ and taking into
account that $f(\alpha)=0$, we have
\begin{equation}\label{b3}
f(x_{n}) = f^{'}(\alpha)(e_{n} +
c_{2}e_{n}^{2}+c_{3}e_{n}^{3}+\ldots+c_{8}e_{n}^{8})+O(e_n^9),
\end{equation}
then
\begin{equation*}
\begin{split}
e_{n,w}&=w_n-\alpha=(1-\beta f'(\alpha))e_n-\beta
f'(\alpha)c_2e_n^2-\beta f'(\alpha)c_3e_n^3-\beta
f'(\alpha)c_4e_n^4-\beta f'(\alpha)c_5e_n^5\\
&+O(e_n^6),
\end{split}
\end{equation*}
and
\begin{equation}\label{b4}
f(w_{n})=f^{'}(\alpha)(e_{n,w} +
c_{2}e_{n,w}^{2}+c_{3}e_{n,w}^{3}+\ldots+c_{8}e_{n,w}^{8})+O(e_{n,w}^9),
\end{equation}
from (\ref{b3}) and (\ref{b4}), we have
\begin{equation}\label{b6}
\begin{split}
e_{n,y}&=y_n-\alpha=(1-\beta f'(\alpha))c_2e_n^2+(-(2+\beta
f'(\alpha)(-2+\beta f'(\alpha)))c_2^2+(-2+\beta
f'(\alpha))\\
&(-1+\beta f'(\alpha))c_3)e_n^3+((4-\beta f'(\alpha)(5+\beta
f'(\alpha)(-3+\beta f'(\alpha))))c_2^3+(-7+\beta
f'(\alpha)\\
&(10+\beta f'(\alpha)(-7+2\beta f'(\alpha))))c_2c_3-(-1+\beta
f'(\alpha))(3+\beta f'(\alpha)(-3+\beta
f'(\alpha)))c_4)e_n^4\\
&+O(e_n^5),
\end{split}
\end{equation}
as well as
\begin{equation}\label{b7}
f(y_{n})=f^{'}(\alpha)(e_{n,y} +
c_{2}e_{n,y}^{2}+c_{3}e_{n,y}^{3}+\ldots+c_{8}e_{n,y}^{8})+O(e_{n,y}^9).
\end{equation}
According to Theorem \ref{t1}, we get
\begin{equation*}
\begin{split}
e_{n,z}&=z_n-\alpha=-c_2(1-\beta f'(\alpha))^2((-1-\beta
f'(\alpha)+2\gamma^2(-1+\beta
f'(\alpha)))c_2^2+c_3)e_n^4\\
&-(-1+\beta f'(\alpha))((-4+\beta f'(\alpha)-(\beta
f'(\alpha))^3-2\gamma(-1+\beta f'(\alpha))^3+2\gamma^3\\
&(-1+\beta f'(\alpha))^3+2\gamma^2(-1+\beta f'(\alpha))(7+2\beta
f'(\alpha)(-3+\beta
f'(\alpha))))c_2^4+(8-6\gamma^2\\
&(-2+\beta f'(\alpha))(-1+\beta f'(\alpha))^2+\beta
f'(\alpha)(-4+\beta f'(\alpha)(-5+3\beta
f'(\alpha))))c_2c_3\\
&-(-2+\beta f'(\alpha))(-1+\beta f'(\alpha))c_3^2-(-2+\beta
f'(\alpha))(-1+\beta f'(\alpha))c_2c_4)e_n^5+O(e_{n}^6),
\end{split}
\end{equation*}
by using the expansion of $f(z_n)$, we have
\begin{equation}\label{b8}
f(z_{n}) = f^{'}(\alpha)(e_{n,z} +
c_{2}e_{n,z}^{2}+c_{3}e_{n,z}^{3}+\ldots+c_{8}e_{n,z}^{8})+O(e_{n,z}^9).
\end{equation}
From (\ref{b7}) and (\ref{b8}), we get
\begin{equation}\label{b9}
\begin{split}
f[z_n,y_n]&=\dfrac{f(z_n)-f(y_n)}{z_n-y_n}=f'(\alpha)+f'(\alpha)c_2^2(1-\beta
f'(\alpha))e_n^2\\
&+(-f'(\alpha)(2+\beta f'(\alpha))c_2^3+f'(\alpha)(-2+\beta
f'(\alpha))(-1+\beta f'(\alpha))c_2c_3)e_n^3\\
&+O(e_n^4).
\end{split}
\end{equation}
And from (\ref{b3}) and (\ref{b7}), we have
\begin{equation}\label{b10}
\begin{split}
f[y_n,x_n]&=\dfrac{f(y_n)-f(x_n)}{y_n-x_n}=f'(\alpha)+c_2f'(\alpha)e_n+f'(\alpha)(c_2^2(1-\beta
f'(\alpha)+c_3)e_n^2 \\
&+f'(\alpha)(-(2+\beta f'(\alpha)(-2+\beta
f'(\alpha)))c_2^3+(-3+\beta f'(\alpha))(-1+\beta
f'(\alpha))c_2c_3+c_4)e_n^3\\
&+O(e_n^4).
\end{split}
\end{equation}
From (\ref{b3}) and (\ref{b4}), we obtain
\begin{equation}\label{b11}
\begin{split}
f[x_n,w_n]&=\dfrac{f(x_n)-f(w_n)}{x_n-w_n}=f'(\alpha)+c_2f'(\alpha)(2-\beta
f'(\alpha))e_n+f'(\alpha)(-c_2^2\beta f'(\alpha)\\
&+(3+c_3\beta f'(\alpha) (-3+\beta f'(\alpha)))e_n^2-f'(\alpha)
(-2 +
   \beta f'(\alpha) ) (-2c_2c_3 \beta f'(\alpha)\\
   &+ (2 +\beta f'(\alpha) (-2 + \beta f'(\alpha)))c_4
   )e_n^3+O(e_n^4).
\end{split}
\end{equation}
From (\ref{b9}) and (\ref{b10}), we get
\begin{equation}\label{b12}
\begin{split}
f[z_n,y_n,x_n]&=\dfrac{f[z_n,y_n]-f[y_n,x_n]}{z_n-x_n}=c_2f'(\alpha)+c_3f'(\alpha)
e_n+f'(\alpha)(c_2c_3(1-\beta f'(\alpha))+c_4)e_n^2\\
&+f'(\alpha)(-c_2^2c_3(2+\beta f'(\alpha)(-2+\beta
f'(\alpha)))+c_3^2(-2+\beta f'(\alpha))(-1+\beta
f'(\alpha))\\
&+c_2c_4(1-\beta f'(\alpha))+c_5)e_n^3+O(e_n^4).
\end{split}
\end{equation}
And from (\ref{b10}) and (\ref{b11}), we get
\begin{equation}\label{b13}
\begin{split}
f[y_n,x_n,w_n]&=\dfrac{f[y_n,x_n]-f[x_n,w_n]}{y_n-w_n}=c_2f'(\alpha)
 +c_3 f'(\alpha) (2-\beta f'(\alpha)) e_n\\
 &+ c_2c_3 f'(\alpha)((1- 2\beta f'(\alpha))+(3+c_4\beta
f'(\alpha) (-3 + \beta f'(\alpha)))) e_n^2\\
&+O(e_n^3).
\end{split}
\end{equation}
Therefore, from (\ref{b12}) and (\ref{b13}), we get
\begin{equation}\label{b14}
\begin{split}
f[z_n,y_n,x_n,w_n]&=\dfrac{f[z_n,y_n,x_n]-f[y_n,x_n,w_n]}{z_n-w_n}=c_3f'(\alpha)
 +c_4 f'(\alpha) (2-\beta f'(\alpha)) e_n\\
 &+ f'(\alpha)(c_2c_4(1-2\beta f(\alpha))+(3+c_5\beta
f'(\alpha)(-3+\beta f'(\alpha)))e_n^2\\
&+O(e_n^3).
\end{split}
\end{equation}
Finally, by substituting (\ref{b3})-(\ref{b14}) in (\ref{b2}), we
obtain
\begin{equation*}
\begin{split}
e_{n+1}&=x_{n+1}-\alpha=(1-\beta f'(\alpha))^4c_2^2(c_2^2(-1-\beta
f'(\alpha)-2\gamma^2(1-\beta f'(\alpha))+c_3)\\
&(c_2^3(-1-\beta f'(\alpha)-2\gamma^2(1-\beta
f'(\alpha)))+c_2c_3-c_4)e_n^8+O(e_n^9).
\end{split}
\end{equation*}
Which shows that the method (\ref{b2}) has optimal convergence
order equal to eight.
\end{proof}

\section{The development of a new method with memory}

In this section, we design a new method with memory by using
\emph{self-accelerating parameters} on method (\ref{b2}). We
observe that the order of convergence of method (\ref{b2}) is
eight when $\beta\neq1/f'(\alpha)$. If $\beta=1/f'(\alpha)$ the
convergence order of method (\ref{b2}) would be 12. Since the
value $f'(\alpha)$ is not available, we use an approximation
$\widehat{f'}(\alpha)\approx f'(\alpha)$, instead. The goal is to
construct a method with memory that incorporates the calculation
of the parameter $\beta=\beta_n$ as the iteration proceeds by the
formula $\beta_n=1/\widehat{f'}(\alpha)$ for $n=1,2,3,\ldots$. It
is assumed that an initial estimate $\beta_0$ should be chosen
before starting the iterative process. In the following, we use
the symbols $\rightarrow$, $O$ and $\sim$ according to the
following convention \cite{Traub}: If $\lim_{n\rightarrow
\infty}f(x_n)=C$, we write $f(x_n)\rightarrow C$ or $f\rightarrow
C$, where $C$ is a nonzero constant. If $\tfrac{f}{g}\rightarrow
C$, we shall write $f=O(g)$ or $f\sim C(g)$.

We approximate $\widehat{f'}(\alpha)$ by $N'_4(x_n)$, then we have
\begin{equation}\label{bbb}
\beta_n=\dfrac{1}{N'_4(x_n)},
\end{equation}
where $N_4'(t):=N_4(t;x_n, z_{n-1}, y_{n-1}, w_{n-1}, x_{n-1})$ is
Newton's interpolation polynomial of fourth degree, set through
five available approximations $(x_n, z_{n-1}, y_{n-1}, w_{n-1},
x_{n-1})$.
\begin{equation}
\begin{split}
N'_4(x_n)=\left[\frac{d}{dt}N_4(t)\right]_{t=x_n}&=f[x_n,z_{n-1}]+f[x_n,z_{n-1},y_{n-1}](x_n-z_{n-1})\\
&+f[x_n,z_{n-1},y_{n-1},w_{n-1}](x_n-z_{n-1})(x_n-y_{n-1})\\
&+f[x_n,z_{n-1},y_{n-1},w_{n-1},x_{n-1}](x_n,z_{n-1})(x_n-y_{n-1})(x_n-w_{n-1}).
\end{split}
\end{equation}
\begin{lem}[\cite{Sharma2}]\label{l1}
If $\beta_n=\frac{1}{N_4'(x_n)}$, $n=1,2,\ldots$ then the estimate
\begin{equation}
(1-\beta_n f'(\alpha))\sim c_5e_{n-1}e_{n-1,w}e_{n-1,y}e_{n-1,z},
\end{equation}
holds.
\end{lem}
\begin{thm}
If $x_0$ is close to a simple zero of the function $f$, then the
convergence R-order of the method (\ref{b2}) with memory with the
corresponding expression (\ref{bbb}) of $\beta_n$ is at least 12.
\end{thm}
\begin{proof}
Let $(x_n)$ be a sequence of approximations produced by an
iterative method (\ref{b2}). If $f(\alpha)=0$ and this sequence
converges to $\alpha $ with the R-order $Q_R((\ref{b2}),
\alpha)\geq r$, we will write
\begin{equation}
e_{n+1}\sim D_{n,r} e_n^r, \quad e_n=x_n-\alpha,
\end{equation}
where $D_{n,r}$ refers to the asymptotic error of (\ref{b2}) when $n$ tend to infinity.
In other words, we get
\begin{equation}\label{f1}
 e_{n+1}\sim D_{n,r}(D_{n-1,r} e_{n-1}^r)^r = D_{n,r} D_{n-1,r}^r
e_{n-1}^{r^2}.
\end{equation}
Assume that the sequences $( w_n) $,  $( y_n)$ and
$(z_n)$ have the R-order $q$, $p$ and $s$, respectively, that
is
\begin{equation}\label{f2}
e_{n,w}\sim D_{n,q}e_n^q \sim
D_{n,q}(D_{n-1,r}e_{n-1}^r)^q=D_{n,q}D_{n-1,r}^q e_{n-1}^{rq},
\end{equation}
\begin{equation}\label{f3}
e_{n,y}\sim D_{n,p}e_n^p \sim
D_{n,p}(D_{n-1,r}e_{n-1}^r)^p=D_{n,p}D_{n-1,r}^p e_{n-1}^{rp},
\end{equation}
and
\begin{equation}\label{f4}
e_{n,z}\sim D_{n,s}e_n^s \sim
D_{n,s}(D_{n-1,r}e_{n-1}^r)^s=D_{n,s}D_{n-1,r}^s e_{n-1}^{rs}.
\end{equation}
Also, we have \\
\\
$e_{n,w}\sim \left(1-\beta_n f'(\alpha)\right)e_n,$
\\
\\
$e_{n,y}\sim c_2\left(1-\beta_n f'(\alpha)\right)e_n^2,$
\\
\\
$e_{n,z}\sim B\left(1-\beta_n f'(\alpha)\right)^2e_n^4,$\\
\\
where $B=-c_2(c_2^2(-1-\beta_n f'(\alpha)-2\gamma^2(1-\beta_n
f'(\alpha)))+c_3),$
\\
\\
$e_{n+1}\sim A(1-\beta_n f'(\alpha))^4e_n^8,$\\
\\
where $A=c_2^2(c_2^2(-1-\beta_n f'(\alpha)-2\gamma^2(1-\beta_n
f'(\alpha))+c_3) (c_2^3(-1-\beta_n f'(\alpha)-2\gamma^2(1-\beta_n
f'(\alpha))+c_2c_3-c_4)$.
Therefore by Lemma \ref{l1}, we obtain
\begin{equation}\label{g2}
\begin{split}
e_{n,w}&\sim (1-\beta_n f'(\alpha))e_n\sim
(c_5e_{n-1}e_{n-1,w}e_{n-1,y}e_{n-1,z})e_n\\
&\sim c_5D_{n-1,q}D_{n-1,p}D_{n-1,s}D_{n-1,r}e_{n-1}^{r+p+s+q+1},
\end{split}
\end{equation}
\begin{equation}\label{g3}
\begin{split}
e_{n,y}&\sim c_2(1-\beta_n f'(\alpha))e_n^2\sim
c_2(c_5e_{n-1}e_{n-1,w}e_{n-1,y}e_{n-1,z})e_n^2\\
&\sim c_2c_5D_{n-1,q}D_{n-1,p}D_{n-1,s}D_{n-1,r}^2
e_{n-1}^{2r+s+p+q+1},
\end{split}
\end{equation}
\begin{equation}\label{g4}
\begin{split}
e_{n,z}&\sim B (1-\beta_n f'(\alpha))^2e_n^4\sim
B(c_5e_{n-1}e_{n-1,w}e_{n-1,y}e_{n-1,z})^2e_n^4\\
&\sim Bc_5^2D_{n-1,q}D_{n-1,p}^2D_{n-1,s}^2D_{n-1,r}^4
e_{n-1}^{4r+2s+2p+2q+2},
\end{split}
\end{equation}
\begin{equation}\label{g1}
\begin{split}
e_{n+1}&\sim A (1-\beta_n f'(\alpha))^4e_n^8\sim
A(c_5e_{n-1}e_{n-1,w}e_{n-1,y}e_{n-1,z})^4e_n^8\\
& \sim Ac_5^4D_{n-1,q}D_{n-1,p}^4D_{n-1,s}^4D_{n-1,r}^6
e_{n-1}^{8r+4s+4p+4q+4}.
\end{split}
\end{equation}
By comparing exponents of $e_{n-1}$ appearing in two pairs of
relations (\ref{f1}), (\ref{g1}) and (\ref{f2}), (\ref{g2}) and
(\ref{f3}), (\ref{g3}) and (\ref{f4}), (\ref{g4}), we obtain the
nonlinear system of four equations and four unknown $r$, $s$, $p$
and $q$.
\begin{equation*}
\begin{cases}
&r^2-8r-4s-4p-4q-4=0,\\
&rs-4r-2s-2p-2q-2=0,\\
&rp-2r-s-p-q-1=0,\\
&rq-r-s-p-q-1=0.
\end{cases}
\end{equation*}
A non-trivial solution of the above system is $ r=12$, $s=6$,
$p=3$, $q=2$.

We have proved that the convergence order of the iterative method
(\ref{b2}) is at least 12.
\end{proof}

\section{Numerical example}

In this section, we show the results of some numerical tests to
compare the efficiencies of methods, using the programming package
Mathematica. To obtain very high accuracy and avoid the loss of
significant digits, we employed multi-precision arithmetic with
1200 significant decimal digits in the programming package
Mathematica 8.

In what follows, we present some concrete iterative methods from the scheme
(\ref{b2}).

\textbf{Method 1.} Choose the weight function $G$ as follows:
\begin{equation}\label{m11}
G(t_n)=1-t_n,
\end{equation}
where $t_n=\frac{f(y_n)}{f(x_n)}$ . The function $G$ in
(\ref{m11}) satisfies the assumptions of Theorem \ref{t2}, then we
have the following method
\begin{equation}\label{m1}
\begin{cases}
y_n=x_n-\dfrac{\beta_n f(x_n)^2}{f(x_n)-f(w_n)},\quad w_n=x_n-\beta_nf(x_n), \quad \beta_n=\frac{1}{N'_4(x_n)},\\
z_n=y_n-\dfrac{f(x_n)+\gamma
f(y_n)}{f(x_n)+(\gamma-2)f(y_n)}\cdot\dfrac{f(x_n)-f(y_n)}{f(x_n)}\cdot\dfrac{f(y_n)}{f[y_n,w_n]},\\
x_{n+1}=z_n-\dfrac{f(z_n)}{f[z_n,y_n]+f[z_n,y_n,x_n](z_n-y_n)+f[z_n,y_n,x_n,w_n](z_n-y_n)(z_n-x_n)},\\
\end{cases}
\end{equation}
where in general, the divide differential of order $n$ is obtained
as:\\
$f[x_0,x_1,\ldots,x_n]=\dfrac{f[x_1,x_2,\ldots
x_n]-f[x_0,x_1,\ldots,x_{n-1}]}{x_n-x_0}$ for $n=1,2,\ldots$.

\textbf{Method 2.} Choose the weight function $G$ as follows:
\begin{equation}\label{m22}
G(t_n)=1-\dfrac{t_n}{1+t_n},
\end{equation}
where $t_n=\frac{f(y_n)}{f(x_n)}$ . The function $G$ in
(\ref{m22}) satisfies the assumptions of Theorem \ref{t2}, then we
have the following method
\begin{equation}\label{m2}
\begin{cases}
y_n=x_n-\dfrac{\beta_n f(x_n)^2}{f(x_n)-f(w_n)},\quad w_n=x_n-\beta_nf(x_n), \quad \beta_n=\frac{1}{N'_4(x_n)},\\
z_n=y_n-\dfrac{f(x_n)+\gamma
f(y_n)}{f(x_n)+(\gamma-2)f(y_n)}\cdot\dfrac{f(x_n)}{f(x_n)+f(y_n)}\cdot\dfrac{f(y_n)}{f[y_n,w_n]},\\
x_{n+1}=z_n-\dfrac{f(z_n)}{f[z_n,y_n]+f[z_n,y_n,x_n](z_n-y_n)+f[z_n,y_n,x_n,w_n](z_n-y_n)(z_n-x_n)}.\\
\end{cases}
\end{equation}
\textbf{Method 3.} Choose the weight function $G$ as follows:
\begin{equation}\label{m33}
G(t_n)=\dfrac{1-2t_n}{1-t_n},
\end{equation}
where $t_n=\frac{f(y_n)}{f(x_n)}$ . The function $G$ in
(\ref{m33}) satisfies the assumptions of Theorem \ref{t2}, then we
have the following method
\begin{equation}\label{m3}
\begin{cases}
y_n=x_n-\dfrac{\beta_n f(x_n)^2}{f(x_n)-f(w_n)},\quad w_n=x_n-\beta_nf(x_n), \quad \beta_n=\frac{1}{N'_4(x_n)},\\
z_n=y_n-\dfrac{f(x_n)+\gamma
f(y_n)}{f(x_n)+(\gamma-2)f(y_n)}\cdot\dfrac{f(x_n)-2f(y_n)}{f(x_n)-f(y_n)}\cdot\dfrac{f(y_n)}{f[y_n,w_n]},\\
x_{n+1}=z_n-\dfrac{f(z_n)}{f[z_n,y_n]+f[z_n,y_n,x_n](z_n-y_n)+f[z_n,y_n,x_n,w_n](z_n-y_n)(z_n-x_n)}.\\
\end{cases}
\end{equation}
\textbf{Method 4.} Choose the weight function $G$ as follows:
\begin{equation}\label{m44}
G(t_n)=(1-t_n)^{\frac{2t_n+1}{t_n+1}},
\end{equation}
where $t_n=\frac{f(y_n)}{f(x_n)}$ . The function $G$ in
(\ref{m44}) satisfies the assumptions of Theorem \ref{t2}, then we
have the following method
\begin{equation}\label{m4}
\begin{cases}
y_n=x_n-\dfrac{\beta_n f(x_n)^2}{f(x_n)-f(w_n)},\quad w_n=x_n-\beta_nf(x_n), \quad \beta_n=\frac{1}{N'_4(x_n)},\\
z_n=y_n-\dfrac{f(x_n)+\gamma
f(y_n)}{f(x_n)+(\gamma-2)f(y_n)}\cdot\left(\dfrac{f(x_n)-f(y_n)}{f(x_n)}\right)^{\frac{2f(y_n)+f(x_n)}{f(y_n)+f(x_n)}}\cdot\dfrac{f(y_n)}{f[y_n,w_n]},\\
x_{n+1}=z_n-\dfrac{f(z_n)}{f[z_n,y_n]+f[z_n,y_n,x_n](z_n-y_n)+f[z_n,y_n,x_n,w_n](z_n-y_n)(z_n-x_n)}.\\
\end{cases}
\end{equation}

The new proposed methods with memory were compared with existing
three-point methods given in what follows, having the same order
convergence and with the same initial data $x_0$ and $\beta_0$.\\
\textbf{Method 5.} The derivative-free method by Kung and Traub
\cite{Kung} given by
\begin{equation}\label{x1}
\begin{cases}
y_n=x_n-\dfrac{f(x_n)}{f[w_n,x_n]},\quad w_n=x_n+\beta_n f(x_n),\quad \beta_n=\frac{-1}{N'(x_n)},\\
z_n=y_n-\dfrac{f(y_n)f(w_n)}{\left(f(w_n)-f(y_n)\right)f[x_n,y_n]},\\
x_{n+1}=z_n-\dfrac{f(y_n)f(w_n)\left(y_n-x_n+\dfrac{f(x_n)}{f[x_n,z_n]}\right)}{\left(f(y_n)-f(z_n)\right)\left(f(w_n-f(z_n)\right)}+\dfrac{f(y_n)}{f[y_n,z_n]}.
\end{cases}
\end{equation}

\textbf{Method 6.} The method by Sharma et al.\ \cite{Sharma2}
given by
\begin{equation}\label{x2}
\begin{cases}
y_n=x_n-\dfrac{ f(x_n)}{\varphi(x_n)},\\
z_n=y_n-H(u_n,v_n)\dfrac{f(y_n)}{\varphi(x_n)},\\
x_{n+1}=z_n-\dfrac{f(z_n)}{f[z_n,y_n]+f[z_n,y_n,x_n](z_n-y_n)+f[z_n,y_n,x_n,w_n](z_n-y_n)(z_n-x_n)},
\end{cases}
\end{equation}
where $\varphi(x_n)=\frac{f(w_n)-f(x_n)}{\beta_n f(x_n)}$,
$w_n=x_n+\beta_n f(x_n)$, $\beta_n=\frac{-1}{N'(x_n)}$,
$H(u_n,v_n)=\dfrac{1+u_n}{1-v_n}$, $u_n=\frac{f(y_n)}{f(x_n)}$ and
$v_n=\frac{f(y_n)}{f(w_n)}$.

\textbf{Method 7.} The method by Zheng et al.\ \cite{Zheng} given
by
\begin{equation}\label{x3}
\begin{cases}
y_{n}=x_{n}-\dfrac{f(x_{n})}{f[x_{n},w_{n}]},\quad w_n=x_n+\beta_n f(x_n),\quad \beta_n=\frac{-1}{N'(x_n)},\\
z_{n}=y_{n}-
\dfrac{f(y_{n})}{f[y_n,x_n]+f[y_n,x_n,w_n](y_n-x_n)},\\
x_{n+1}=z_n-\dfrac{f(z_n)}{f[z_n,y_n]+f[z_n,y_n,x_n](z_n-y_n)+f[z_n,y_n,x_n,w_n](z_n-y_n)(z_n-x_n)},
\end{cases}
\end{equation}


In order to test our proposed methods with memory (\ref{m1}),
(\ref{m2}), (\ref{m3}) and (\ref{m4}) and also compare them with
the methods \eqref{x1}, \eqref{x2} and \eqref{x3}, we choose the
initial value $x_0$ using the \texttt{Mathematica} command
\texttt{FindRoot} \cite[pp.\ 158--160]{Hazrat} and compute the
error and the computational order of convergence (coc) by the
approximate formula \cite{Fer}
\begin{equation*}
\text{coc}\approx\frac{\ln|(x_{n+1}-\alpha)/(x_{n}-\alpha)|}{\ln|(x_{n}-\alpha)/(x_{n-1}-\alpha)|}.
\end{equation*}
\newpage
\begin{table}[h!]
\begin{center}
\begin{tabular}{l c}
  \hline
    Test function $f_n$ & Root $\alpha$
  \\ \hline
    $f_1(x) = \ln (x^2-2x+2)+e^{x^2-5x+4}\sin(x-1)$ & $1$
  \\[0.5ex]
   $f_2(x) = e^{x^2+x \cos(x)-1}\sin(\pi x)+x\ln(x\sin (x)+1)$ & $0$
  \\[0.5ex]
   $ f_3(x)=\left(1-\sin(x^2)\right)\frac{1+x^2}{1+x^3}+x\ln(x^2-\pi +1)-\frac{1+\pi}{1+\sqrt{\pi^3}}$ & $\sqrt{\pi}\ $
    \\ \hline
\end{tabular}
\end{center}
\vspace*{-3ex} \caption{Test functions $f_1, f_2, f_3$ and root
$\alpha$.\label{table1}}
\end{table}
\begin{table}[h!]
\begin{center}
\begin{tabular}{c l l l l }
\hline
\\
 ~~ & Method $(\ref{m1})$ & Method $(\ref{m2})$ & Method $(\ref{m3})$ & Method$(\ref{m4})$ \\
\\
\hline
\\
$f_1$, $x_0=1.35$\\
$\vert x_{1}-\alpha\vert$ & $0.314e-6$ & $0.134e-5$ & $0.543e-6$ & $0.703e-6$ \\
$\vert x_{2}-\alpha\vert$ & $0.110e-66$& $0.153e-61$ & $0.715e-65$ & $0.542e-64$ \\
$\vert x_{3}-\alpha\vert$ & $0.178e-800$ & $0.914e-739$ & $0.100e-788$ & $0.360e-768$\\
$coc$ & $12.1378$ & $12.1056$ & $12.1238$ & $12.1175$\\
\\
\\
$f_2$, $x_0=0.6$\\
$\vert x_{1}-\alpha\vert$ & $0.900e-4$ & $0.585e-3$ & $0.308e-4$ & $0.539e-4$  \\
$\vert x_{2}-\alpha\vert$ & $0.111e-44$& $0.713e-38$ & $0.650e-49$  & $0.335e-48$ \\
$\vert x_{3}-\alpha\vert$ & $0.255e-536$ & $0.126e-454$ & $0.419e-587$ & $0.148e-578$ \\
$coc$ & $12.0178$ & $11.9363$ & $12.0465$ & $11.9972$\\
\\
\\
$f_3$, $x_0=1.7$\\
$\vert x_{1}-\alpha \vert$ & $0.836e-8$ & $0.164e-7$ & $0.605e-8$ & $0.290e-8$ \\
$\vert x_{2}-\alpha \vert$ & $0.102e-94$& $0.101e-96$ & $0.624e-97$ & $0.116e-100$ \\
$\vert x_{3}-\alpha \vert$ & $0.134e-1138$ & $0.118e-1090$ & $0.341e-1165$ & $0.612e-1210$\\
$coc$ & $12.0110$ & $12.0173$ & $12.0048$ & $12.0057$ \\
\hline
\end{tabular}
\end{center}
\vspace*{-3ex} \caption{Errors and coc for methods (\ref{m1}),
(\ref{m2}), (\ref{m3}) and (\ref{m4}) with $\gamma=0$.
\label{table2}}
\end{table}

\newpage
\begin{table}[h!]
\begin{center}
\begin{tabular}{c l l l }
\hline
\\
 ~~ & Method $(\ref{x1})$ & Method $(\ref{x2})$ & Method $(\ref{x3})$  \\
\\
\hline
\\
$f_1$, $x_0=1.35$\\
$\vert x_{1}-\alpha \vert$ & $0.845e-4$ & $0.308e-6$ & $0.148e-5$  \\
$\vert x_{2}-\alpha \vert$ & $0.393e-45$& $0.179e-67$ & $0.157e-61$  \\
$\vert x_{3}-\alpha \vert$ & $0.100e-540$ & $0.126e-812$ & $0.481e-738$ \\
$coc$ & $11.9906$ & $12.1688$ & $12.0973$ \\
\\
\\
$f_2$, $x_0=0.6$\\
$\vert x_{1}-\alpha \vert$ & $0.798e-3$ & $0.891e-4$ & $0.214e-4$   \\
$\vert x_{2}-\alpha \vert$ & $0.194e-40$& $0.541e-45$ & $0.168e-53$   \\
$\vert x_{3}-\alpha \vert$ & $0.976e-486$ & $0.274e-543$ & $0.386e-642$  \\
$coc$ & $11.8387$ & $12.0827$ & $11.9875$ \\
\\
\\
$f_3$, $x_0=1.7$\\
$\vert x_{1}-\alpha \vert$ & $0.241e-8$ & $0.757e-8$ & $0.221e-7$  \\
$\vert x_{2}-\alpha \vert$ & $0.137e-99$& $0.267e-96$ & $0.140e-90$  \\
$\vert x_{3}-\alpha \vert$ & $0.283e-1196$ & $0.561e-1158$ & $0.546e-1089$ \\
$coc$ & $12.0190$ & $12.0028$ & $12.0001$  \\
\hline
\end{tabular}
\end{center}
\vspace*{-3ex} \caption{Errors and coc for methods (\ref{x1}),
(\ref{x2}) and (\ref{x3}). \label{table3}}
\end{table}
In Table \ref{table2}, the proposed methods with memory with
weight functions (\ref{m11}), (\ref{m22}), (\ref{m33}) and
(\ref{m44}) and in Tabel \ref{table3} the methods
(\ref{x1})-(\ref{x3}) have been tested on three different
nonlinear equations which were introduced in Table {\ref{table1}.
It is clear that these methods are in accordance with the
developed theory.
\section{Conclusion}
We have introduced a new method with memory for approximating a
simple root of a given nonlinear equation. An increase of the
convergence order is attained without any additional function
evaluations, which points to a very high computational efficiency
of the proposed methods with memory. We have proved that the
convergence order of the new method with memory is, at least 12.
So its efficiency index is $12^{1/4}=1.86121$ which is greater
than that of the three-point methods of order eight
$8^{1/4}=1.68179$ with the same function evaluations. Numerical
examples show that our methods with memory work and can compete
with other methods under the same conditions.


\end{document}